\documentclass[journal,twoside,web]{ieeecolor} 
\usepackage{lcsys}
\usepackage{cite}
\usepackage{graphicx}
\usepackage{amsfonts}
\usepackage{amsmath} 
\usepackage{amssymb}
\usepackage{dsfont}
\usepackage{enumerate}
\usepackage{mathrsfs}
\usepackage{cleveref}
\usepackage[labelformat=simple]{subcaption}

\newtheorem{thm}{Theorem}
\newtheorem{lem}[thm]{Lemma}

\newtheorem{cor}{Corollary}
\newtheorem{defn}{Definition}

\newtheorem{rem}{Remark}
\newtheorem{ass}{Assumption}

\newcommand{\R}{{\mathbb R}}

\newcommand{\prob}[1]{\mathop{{\rm Pr}} \left(#1 \right)}

\newcommand{\Rset}{\mathbb{R}}

\newcommand{\Zset}{\mathbb{Z}}

\newcommand{\Xcal}{{\cal X}}

\newcounter{l1}
\newcounter{l2}
\newcounter{l3}
\newcommand{\bdotlist}{\begin{list}{$\bullet$}{}}
\newcommand{\bboxlist}{\begin{list}{$\Box$}{}}
\newcommand{\bbboxlist}{\begin{list}{\raisebox{.005in}{{\tiny
$\blacksquare$ \ \ }}}{}}
\newcommand{\bdashlist}{\begin{list}{$-$}{} }
\newcommand{\blist}{\begin{list}{}{} }
\newcommand{\barablist}{\begin{list}{\arabic{l1}}{\usecounter{l1}}}
\newcommand{\balphlist}{\begin{list}{(\alph{l2})}{\usecounter{l2}}}
\newcommand{\bAlphlist}{\begin{list}{\Alph{l2}.}{\usecounter{l2}}}
\newcommand{\bdiamlist}{\begin{list}{$\diamond$}{}}
\newcommand{\bromalist}{\begin{list}{(\roman{l3})}{\usecounter{l3}}}

\newcommand{\rand}{\boldsymbol{\xi}}
\newcommand{\randa}{\boldsymbol{\eta}}
\newcommand{\reliz}{\xi}

\newcommand{\risk}{\epsilon}

\newcommand{\indi}[1]{\mathds{1} \left\{#1 \right\}}
\newcommand{\dist}{\mathbb{P}}
\newcommand{\wass}[2]{d_{\rm W}\left( #1, \, #2 \right)}

\newcommand{\saaset}{\widehat{\Xcal}_{\alpha, \gamma}}
\newcommand{\ccset}{\Xcal_{\risk}}
\newcommand{\robfinset}{\widehat{\Xcal}^r_{\alpha}}
\newcommand{\robfinviol}{\widehat{v}^r}
\newcommand{\robinfset}{\widehat{\Xcal}^r_{\alpha,\gamma}}
\newcommand{\robinfviol}{\widehat{v}^r_{\gamma}}
\newcommand{\uncertset}{U_{r_i}(\rand_i)}

\newcommand{\comp}{\overline{\Xcal}_{\risk}}
\newcommand{\viol}[1]{\widehat{v}_{\gamma} (#1)}

\begin{document}

\title{Data-Driven Approximations of Chance Constrained Programs in Nonstationary Environments}

\author{Shuhao Yan, \ Francesca Parise, \ Eilyan Bitar
\thanks{This research was supported by The Nature Conservancy and  the Cornell  Atkinson Center for Sustainability.}
\thanks{The authors are with the School of Electrical and Computer Engineering, Cornell University, Ithaca, NY, 14853, USA. }
}

\maketitle 
\thispagestyle{empty} 

\begin{abstract} 
We study sample average approximations (SAA) of chance constrained programs. SAA methods typically approximate the actual distribution in the chance constraint using an empirical distribution 
constructed from random samples assumed to be independent and identically distributed according to the actual distribution. 
In this paper, we consider a nonstationary variant of this problem,
where the random samples are assumed to be independently drawn in a sequential fashion from an unknown and possibly time-varying distribution. 
This nonstationarity  may be driven by changing environmental conditions present in many real-world applications.
To account for the potential nonstationarity in the data generation process,  we propose a novel robust SAA method  exploiting
information about the Wasserstein distance between the sequence of data-generating distributions and the actual chance constraint distribution. 
As a key result, we obtain distribution-free estimates of the sample size required to ensure that the robust SAA method will yield solutions that are feasible for the chance constraint under the actual distribution with high confidence. 
\end{abstract}


\section{Introduction} \label{sec:introduction}
\IEEEPARstart{W}{e} consider a class of  chance constrained optimization problems  of the form
\begin{align} \label{eq:chance}
\underset{x \in \Xcal}{\text{minimize}} \ f(x)  \  \text{subject  to} \   \prob{g(x,  \rand ) \leq 0} \geq 1 - \risk.
\end{align}
Here, $x \in \Rset^n$ denotes the optimization variable, $\Xcal \subset \Rset^n$ is a deterministic set,  $f:\Rset^n \rightarrow \Rset$ is  the objective function,  $g: \Rset^n \times \Rset^d \rightarrow \Rset$ is a given  constraint function,  and $\rand$ is an $\Rset^d$-valued  random vector whose probability distribution is supported on a set $\Xi \subseteq \Rset^d$. The random vector  reflects uncertainty in the   constraint function, and  the risk  parameter $\risk \in [0,1]$ encodes the maximum   probability of constraint violation that the decision maker is willing to tolerate.

Chance constrained programs (CCPs) can be difficult to solve for several reasons. First, chance constraints typically result in nonconvex feasible sets \cite{ahmed2008solving}, and have been shown to result in NP-hard optimization problems in  certain settings  \cite{luedtke2010integer}. Second, the underlying probability distribution of the random variable $\rand$ is often unknown. These issues render the exact solution of CCPs intractable, and motivate the development of data-driven approximation methods.

Well-known approximation techniques include the sample average approximation (SAA) method \cite{luedtke2008sample} and the scenario approach \cite{campi2008exact, Calafiore2006TAC_scenario}. Both methods approximate the original chance constraint using randomly sampled constraints, which are based on samples assumed to be drawn in an independent and identically distributed (i.i.d.) fashion from the \emph{actual} distribution used in the definition of the  original chance constraint.
However, in many real-world applications where the underlying data-generating processes may be nonstationary in nature, it may not be possible to obtain  i.i.d. samples from the actual distribution of interest. 
This nonstationarity may be driven by gradual sensor degradation over time, sudden hardware faults, changing environmental conditions, or shifting users behavior \cite{Ditzler2015nonstationary}. 
Many distributionally robust optimization methods  \cite{delage2010distributionally, wiesemann2014distributionally, esfahani2018data,hu2013kullback, nemirovski2007convex, tseng2016random, Shapiro2017_distribution_robust, erdougan2006ambiguous} have been developed to compensate for potential discrepancies between the data-generating distribution and the actual distribution of interest. 
Importantly, these methods rely on the assumption that the random samples are drawn from a \emph{common} distribution.
This assumption may fail to hold in nonstationary environments, where the random samples are drawn from possibly different distributions over time.\footnote{Although one could  apply the distributionally robust approach in nonstationary environments by considering the worst-case discrepancy between the sequence of data-generating distributions and the chance constraint distribution, such a procedure is likely  overly conservative.}

\subsubsection*{Summary of Contributions} This paper addresses this gap by developing SAA schemes tailored to nonstationary environments. 
We begin by studying SAA schemes in stationary environments. 
As a first contribution, we establish a new bound on the probability that a feasible solution to the SAA violates the original chance constraint, improving upon the previously best-known bound \cite[Theorem 10]{luedtke2008sample}. 
As a second contribution, we  turn to nonstationary environments and propose a novel robust SAA scheme utilizing  information about the distance between the sequence of unknown data-generating distributions and the actual chance constraint distribution. 
While the classic SAA scheme only enforces sampled constraints at the random samples, the key novelty of our approach is to require that each sampled constraint be enforced robustly with respect to an uncertainty set defined as the intersection of the support set $\Xi$ and a norm-ball centered at each random sample $\rand_i$  with a radius $r_i$. The radius $r_i$ depends explicitly on the distance between the data-generating distribution $\dist_i$ and the chance constraint distribution.
The idea of imposing constraints robustly around the samples is also used in distributionally robust optimization methods (e.g. \cite{erdougan2006ambiguous,tseng2016random}) to approximate ambiguous CCPs by robust sampled programs. 
An advantage of our approach is that different radii can be used for different samples to capture variations in data-generating distributions. By using the Wasserstein metric to quantify the distance between distributions, we provide upper bounds on the probability that a feasible solution to the proposed robust SAA violates the original chance constraint (which is termed the probability of infeasibility), as a function of the radii. 
We first derive our results for sets $\mathcal{X}$ with finite cardinality and then suggest an extension of the proposed approach to generic bounded sets $\mathcal{X}$, under an additional Lipschitz assumption on the given constraint function.

\subsubsection*{Organization} The remainder of this paper is organized as follows. 
In Section \ref{sec:SAA}, we recap results pertaining to sample average approximations of CCPs in stationary environments, present a new and improved upper bound on the probability of infeasibility, and compare our bound with the previously best-known bound \cite[Theorem 10]{luedtke2008sample}.
In Section \ref{sec:nonstat}, we introduce the nonstationary data-generation model, propose a robust SAA scheme tailored to this setting, and derive upper bounds on the corresponding probability of infeasibility.
 Section \ref{sec:conclusion} concludes the paper.

\subsubsection*{Notation} 
 Let $\Rset$, $\Rset_+$ and $\Zset$ denote the sets of real numbers, nonnegative real numbers and  integers, respectively. Given a positive integer $n \in \Zset$, we let $[n] := \{1, \dots, n\}$ denote the set of the first $n$ integers. 
 Given a real number $x \in \Rset$, we denote its \emph{ceiling} by $\lceil x \rceil : = \min\{ n \in \Zset \ | \ n \geq x \}$, its \emph{floor} by $\lfloor x \rfloor : = \max\{ n \in \Zset \ | \ n \leq x \}$, and its \emph{positive part} by $(x)_+:=\max\{0,x\}$. 
 We use boldface symbols to denote random variables, and non-boldface symbols to denote particular values in the range of a random variable and other deterministic quantities. We  let $\prob{A}$ denote the probability of an event $A$, and $\mathbb{E}[\rand]$  denote the expected value of a random variable $\rand$. The indicator function is denoted by $\indi{\cdot}$.
\section{Sample Average Approximation  in Stationary Environments} \label{sec:SAA}
We first rewrite the chance constrained  problem \eqref{eq:chance} as
\begin{align*}
\underset{x \in \Xcal}{\text{minimize}} \ f(x)  \  \text{subject  to} \   v(x) \leq  \risk,
\end{align*}
where $ v(x) :=     \prob{g(x,  \rand ) > 0}$ denotes the \emph{constraint violation probability} at a point $x \in \Xcal$. To ensure that the function $v(x)$ is well defined, the function $g(x,  \cdot): \Xi \rightarrow \Rset$ is assumed to be measurable for every $x \in \Rset^n$.  
The feasible region of problem \eqref{eq:chance} is denoted by
\begin{align*}
\Xcal_{\risk} := \{ x \in \Xcal \ | \  v(x) \leq  \risk\}.
\end{align*}

In this section, we consider a stationary data-generating environment and assume that i.i.d. random samples $\rand_1, \rand_2,\dots$ can be drawn from the \emph{actual} distribution specifying the chance constraint. Based on these samples, one can  approximate the  chance constraint by replacing the data-generating distribution with the empirical distribution. This results in an empirical approximation of $v(x)$ as
\begin{align*}
\viol{x} := \frac{1}{N} \sum_{i=1}^N \indi{ g(x, \rand_i) + \gamma > 0 },
\end{align*}
where $N$ is the sample size and $\gamma \in \Rset_+$ is a constraint tightening parameter. 
Using this empirical approximation of the constraint violation probability, Luedtke and Ahmed \cite{luedtke2008sample} define the \emph{sample average approximation} of 
problem \eqref{eq:chance} as
\begin{align} \label{eq:SAA}
\underset{x \in \Xcal}{\text{minimize}} \ f(x)  \  \text{subject  to} \   \viol{x} \leq  \alpha,
\end{align}
where the fixed risk level $\alpha \in [0,1]$ is a design parameter and   may be chosen to be smaller than  the original risk level $\risk$  to compensate for the discrepancy between the actual and empirical distribution.
Overall, the parameters $\alpha$ and $\gamma$ require the \emph{strict} satisfaction of at least $\lceil (1-\alpha)N \rceil$ of the sampled constraints with a margin $\gamma$.\footnote{For $ \alpha = \gamma = 0$,  the SAA \eqref{eq:SAA} reduces to a scenario approximation of the CCP, requiring that all $N$ sampled constraints be satisfied (e.g., \cite{campi2008exact, calafiore2010random}).}

It is  reasonable to expect that, for risk levels $0 \leq \alpha < \risk$,   any feasible solution to the SAA \eqref{eq:SAA} will  be feasible for problem \eqref{eq:chance} with high probability given a large enough  sample size $N$.   
To formalize this intuition, we denote the  feasible region of problem \eqref{eq:SAA} by
\begin{align*}
\saaset:= \{ x \in \Xcal \ | \  \viol{x} \leq  \alpha \},
\end{align*}
and  derive bounds on the probability that any feasible solution to the SAA \eqref{eq:SAA} violates the original chance constraint, which we refer to as the \emph{probability of infeasibility}.

\subsection{Finite $\mathcal{X}$} \label{subsec:stationary with finite X}
Before stating our main result of this section, we recall a  known result providing an upper bound on the  \emph{probability of infeasibility}, as defined above, for sets $\mathcal{X}$ of  finite cardinality. 
\begin{thm}[\!\!{\cite[Theorem 5]{luedtke2008sample}}] \label{thm:saa}
Suppose that $|\Xcal| < \infty$. Let $\gamma=0$. Then
$$
\prob{\saaset \nsubseteq \Xcal_\epsilon} \leq |\Xcal|\Phi(\alpha N; \epsilon, N), 
$$
where $$\Phi( z ; \risk, N) := \sum_{i=0}^{ \lfloor  z \rfloor} \binom{N}{i} \risk^i (1- \risk)^{N-i}, \quad z \in [0,N],$$
denotes the cumulative distribution function of a binomial random variable with $N$ trials and success probability $\risk$.
\end{thm}

A similar upper bound using the finite cardinality assumption is also established in \cite[Theorem 2]{Alamo2015auto_Randomized}.

\subsection{Lipschitz Continuous $g$}\label{subsec:infinite X in stationary} 
As  argued in \cite{luedtke2008sample}, it is possible to extend Theorem \ref{thm:saa} to  settings where the set $\Xcal$ is not finite by relying on the following regularity assumptions.
\begin{ass} \label{ass:lip} There exists a constant  $L \in (0, \infty)$ such that $| g(x, \reliz) -  g(y, \reliz) | \leq L \| x - y\|_\infty$ 
for all $x,y \in \Xcal$ and $\reliz \in \Xi$.
\end{ass}
\begin{ass} \label{ass:bound} There exists a constant $D \in (0, \infty)$ such that $\sup \{ \|x-y \|_\infty \, | \, x,y \in \Xcal \} \leq D$.
\end{ass}

Assumptions \ref{ass:lip} and \ref{ass:bound} ensure boundedness of the constraint function $g$  with respect to the optimization variable $x$ over the set $\Xcal$. Using these assumptions, we provide a novel upper bound on the \textit{probability of infeasibility}, 
which improves upon the previously best-known upper bound provided 
in \cite[Theorem 10]{luedtke2008sample} under certain conditions.

\begin{thm} \label{thm:main} Suppose that Assumptions \ref{ass:lip} and \ref{ass:bound} hold. 
Then
\begin{equation}
\prob{\saaset \nsubseteq \ccset} \leq  \left( \frac{LD}{\gamma} +1 \right)^n   \Phi(\alpha N ;  \risk, N). \label{eq:our bound}
\end{equation}
\end{thm}

\begin{proof}
 Let  $\sigma := \gamma / L$ and $\overline{\Xcal}_\risk := \Xcal \setminus \Xcal_\risk$. Since $\overline{\Xcal}_\risk \subseteq \Xcal$,  it follows from Assumption \ref{ass:bound}  that there exists an $n$-dimensional hypercube with  edge length equal to $D$ that contains  the set $\overline{\Xcal}_\risk$. Thus,  there exists an \emph{internal $\sigma$-covering}  $\{\bar{x}_1, \dots, \bar{x}_K\} \subseteq \comp$ of the set $\comp$  that satisfies
$
\comp \subseteq \bigcup_{k=1}^K B_{\sigma}^\infty(\bar{x}_k),
$
where $B_{\sigma}^\infty(x) \subseteq \Rset^n$ denotes the closed $L^\infty$-ball of radius $\sigma$  centered at the point $x\in \R^n$. It follows from classical internal-covering number bounds \cite[Lemma 5.7]{wainwright_2019} that $K \leq  \left( D/\sigma +1 \right)^n.  $ 
Let $\Xcal_k :=  B_{\sigma}^{\infty}(\bar{x}_k) \bigcap \overline{\Xcal}_\risk$ for $k =1, \dots,K$.  It follows that
\begin{align}
\nonumber&\hspace{-.3in} \prob{\saaset \nsubseteq \ccset}\\ 
\nonumber& =  \prob{ \exists \, x \in \comp \ \text{such that} \ \viol{x} \leq \alpha }\\
\label{eq:a1} & \leq \sum_{k=1}^K \prob{ \exists \, x \in \Xcal_k     \ \text{such that} \  \viol{x} \leq \alpha }.
\end{align}
Note that  for any $k \in [K]$, it holds that
\begin{align} 
\nonumber & \hspace{-.3in}\prob{ \exists \, x \in \Xcal_k \ \text{such that} \ \viol{x} \leq \alpha }\\
\nonumber  &\leq  \prob{ \inf_{x \in \Xcal_k} \viol{x} \leq \alpha} \\
\nonumber & =  \prob{  \inf_{x \in \Xcal_k} \, \sum_{i=1}^N  \indi{ g(x, \rand_i) + \gamma > 0  } \leq \alpha N  }\\
\nonumber & \leq \prob{ \sum_{i=1}^N    \inf_{x_i \in \Xcal_k} \indi{  g(x_i, \rand_i) + \gamma > 0  } \leq \alpha N  }\\
\nonumber & = \prob{ \sum_{i=1}^N     \indi{ \inf_{x_i \in \Xcal_k} g(x_i, \rand_i) + \gamma > 0  } \leq \alpha N  } \\
\label{eq:bb} & \leq  \prob{ \sum_{i=1}^N     \indi{  g(\bar{x}_k, \rand_i)  > 0  } \leq \alpha N  }.
\end{align} 
The final inequality is a consequence of Assumption \ref{ass:lip} and the facts that $x_i \in \Xcal_k \subseteq B_{\sigma}^{\infty}(\bar{x}_k)$ and $ \sigma = \gamma / L$, which together  imply that 
$$ | g(x_i, \reliz) -  g(\bar{x}_k, \reliz) | \, \leq \,  L \| x_i - \bar{x}_k\|_\infty  \, \leq \, L \sigma \, = \, \gamma$$
for all $x_i \in \Xcal_k$ and $\reliz \in \Xi$. Therefore, $g(\bar{x}_k,\reliz) \leq g(x_i,\reliz)+\gamma$ for all $x_i \in \Xcal_k$ and $\reliz \in \Xi$.

This implies that 
$
\indi{g(\bar{x}_k,\reliz)>0} \leq \indi{\inf_{x_i \in \Xcal_k} g(x_i,\reliz)+\gamma >0 }
$
for all $\reliz \in \Xi$, which proves  inequality \eqref{eq:bb}.
It also holds that
\begin{align}
\nonumber  \prob{ \sum_{i=1}^N     \indi{  g(\bar{x}_k, \rand_i)  > 0  } \leq \alpha N  }  & = \Phi( \alpha N ; v(\bar{x}_k), N) \\
\label{eq:b1} & \leq  \Phi( \alpha N ; \risk, N).
\end{align}
The  equality follows from the fact $\sum_{i=1}^N     \indi{  g(\bar{x}_k, \rand_i)  > 0  } $ is a binomial random variable with $N$ trials and a success probability $v(\bar{x}_k) = \prob{g(\bar{x}_k, \rand)  > 0}$. The   inequality follows from $v(\bar{x}_k)  > \risk$, which is a consequence of $\bar{x}_k \in \overline{\Xcal}_\risk$. Combining inequalities \eqref{eq:a1}, \eqref{eq:bb}, \eqref{eq:b1}, and the fact that $K \, \leq \, ( LD/ \gamma +1 )^n$, the desired result follows.
\end{proof}

\begin{rem}[Comparison to \cite{luedtke2008sample}] The upper bound on the \emph{probability of infeasibility} given in \cite[Theorem 10]{luedtke2008sample} is 
\begin{equation}
    \prob{\saaset \nsubseteq \Xcal_\risk } \leq \left\lceil \frac{1}{\beta} \right\rceil \left\lceil \frac{2LD}{\gamma} \right\rceil^n \Phi(\alpha N; \epsilon-\beta,N), \label{eq: Luedtke bound}
\end{equation}
where $\beta \in (0, \epsilon)$ is an additional parameter associated with the specific approach  used in \cite{luedtke2008sample} to derive such upper bound.
Since $\Phi(\alpha N; \epsilon,N)\le \Phi(\alpha N; \epsilon-\beta,N)$ and $1<\bigl\lceil \frac{1}{\beta}  \bigr\rceil$ for any $\beta\in (0, \epsilon)$,
it is straightforward to see that the bound \eqref{eq:our bound} provided in this paper strictly improves upon the bound  \eqref{eq: Luedtke bound} if the margin $\gamma$ satisfies $\gamma \leq LD$. In Figure \ref{fig:compare upper bounds},  we provide a graphical comparison of the bounds, which shows that the bound \eqref{eq:our bound} provided in this paper improves upon the bound  \eqref{eq: Luedtke bound} by many orders of magnitude for modest sample sizes~$N$.
\end{rem}
    \begin{figure}[ht]
         \centering 
         \includegraphics[width=0.5\textwidth]{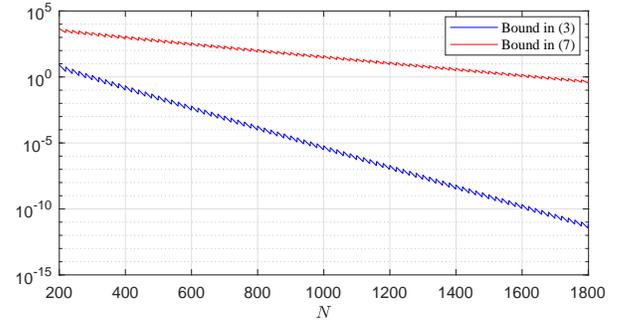}
        \caption{Comparison of  the upper bounds \eqref{eq:our bound} and  \eqref{eq: Luedtke bound}  on the \emph{probability of infeasibility} for $n=10$,  $\risk=0.1$, $\alpha=0.05$, $LD/\gamma=1$, and $\beta=(\risk-\alpha)/2$.}
         \label{fig:compare upper bounds}
     \end{figure}
\section{Robust Sample Average Approximation in Nonstationary Environments}
\label{sec:nonstat}
In this section, we consider a nonstationary variant of the SAA problem, where the random sample is drawn in a sequential fashion from an unknown distribution that may change over time as defined in Section \ref{sec:data}. 
To account for the potential nonstationarity in the data-generating process,  we propose a novel robust SAA in Section \ref{subsec:robust SAA}. Finally, in Sections \ref{subsec:nonstat finite X} and \ref{subsec:nonstat infinite X}, we establish upper bounds on the corresponding \emph{probability of infeasibility} in terms of a target distribution that reflects the current state of the environment.

\subsection{Nonstationary Data-Generation Model} \label{sec:data}
We first introduce the nonstationary data-generation model considered in this paper. Specifically, the random sample $\rand_1, \rand_2, \dots$ is assumed to be independent but not necessarily identically distributed. We denote by  $\dist_i$  the probability measure according to which the random variable $\rand_i$ is defined. All of the probability measures considered are defined on a  common measurable space $(\Xi, \mathcal{B}(\Xi))$, where $\mathcal{B}(\Xi)$ is the Borel sigma-algebra of Borel sets of $\Xi$.  The set $\Xi$ is also assumed to be a Polish space.

The dependence on time is captured by the index of each random variable $\rand_i$, where larger indices correspond to more recently sampled data, and smaller indices correspond to older data.  To constrain the temporal changes in the data-generating distribution over time, we employ the 1-Wasserstein metric, which is defined as follows.
\begin{defn} 
The \emph{1-Wasserstein distance} between two probability measures $\dist$ and $\dist'$ on $\Xi$ is defined as 
\begin{align}
\label{eq:wass_def}
\wass{\dist}{\dist'} \, :=  \underset{\pi \in \Pi(\dist, \dist')}{\inf}   \int_{\Xi \times \Xi}  \|\xi - \xi'\| \, \pi(d\xi, \, d \xi'),
\end{align}
where $\|\cdot\|$ is a norm on $\Xi$, and $\Pi(\dist, \dist')$ denotes the set of all joint probability distributions
of $\rand$ and $\rand'$  with marginal distributions $\dist$ and $\dist'$, respectively. 
\end{defn}

Note that a  joint distribution (coupling) that achieves the infimum in \eqref{eq:wass_def} is guaranteed to exist, as we have assumed that $\Xi$ is a Polish space \cite[Proposition 1.3]{wang2012coupling}. The existence of such couplings will play an instrumental role in the proof of Lemma \ref{lemma:intermediate result}, which is a key  building block in the derivation of the main results of this section.

Using the 1-Wasserstein metric, we constrain the allowable changes in the data-generating process according to
\begin{align}
\label{eq:dist}
\wass{\dist_i}{\dist_{i+k}} \leq \rho(k) 
\end{align}
for all indices $i \geq 1$ and $ k \geq 0$, where $\rho: \Rset_+ \rightarrow \Rset_+$ is a known function satisfying $\rho(0) = 0$.  The function $\rho(k)$ can be understood as a \emph{variation budget}, 
limiting the extent to which the underlying data-generating distribution can change over a given span of $k$ time periods. It allows for a broad range of temporal shifts in the data-generating process, including  gradual drifts over time, large but infrequent changes, or a  combination thereof. 

We note that, while our model uses the 1-Wasserstein metric, the results of this section also hold for alternative probability metrics/distances that imply the satisfaction of condition \eqref{eq:dist} under the 1-Wasserstein metric. We refer the reader to \cite{gibbs2002choosing} for a  survey  of bounds relating different probability metrics/distances to the 1-Wasserstein metric.

\subsection{Robust Sample Average Approximation}\label{subsec:robust SAA}
Given a random sample $\rand_1, \dots, \rand_N$, we are interested in constructing 
a sample average approximation
of the chance constrained feasible region defined in terms of the distribution of the ensuing $(N+1)$-th random variable $\rand_{N+1}$.
We denote the corresponding feasible region by
\begin{align*}
\Xcal_{\risk}^{N+1} : = \{x \in \Xcal \ | \ v^{N+1}(x)  \leq \risk\}, 
\end{align*}
where $v^{N+1}(x) := \prob{g(x, \rand_{N+1})  > 0}$ denotes the \emph{target constraint violation probability} at a point $x \in \Xcal$ under the \emph{target distribution} $\dist_{N+1}$. 
The target distribution  should be understood as reflecting the current state of the environment. To account for the potential discrepancy between the sequence of sampling distributions $\{\dist_i\}_{i=1}^N$ and the \emph{target distribution} $\dist_{N+1}$,  we  consider  a robust empirical estimate of the target constraint violation probability given by 
\begin{equation}
    \robfinviol(x) := \frac{1}{N} \sum_{i=1}^N \indi{\exists \, u \in \uncertset  \text{ such that } g(x,u) >0}, \nonumber
\end{equation}
where 
\vspace{-.1in}
\begin{align} \label{eq:uncertainty set}
\uncertset :=  \{ u \in \Xi  \ | \  \| u- \rand_i \| \leq r_i\}
\end{align}
denotes the intersection of $\Xi$ and a closed norm-ball of radius $r_i \geq 0$ centered at the sample $\rand_i$ for all $i \in [N]$.  The radii $r = (r_1, \dots, r_N) \in \Rset_+^N$ of the norm-balls used in the above approximation are design parameters to be specified by the user. Using the above approximation of the target constraint violation probability, we define a  \emph{robust sample average  approximation} of the feasible region  $\Xcal_\risk^{N+1}$ as
\begin{equation}
    \robfinset:=\{ x \in \Xcal \ | \ \robfinviol(x) \leq \alpha\}. \label{eq:robust saa}
\end{equation}

Note that the robust SAA \eqref{eq:robust saa} reduces to the conventional SAA if the norm-ball radii are all equal to zero.
If, on the other hand, the norm-balls have nonzero radii, then the robust SAA requires the satisfaction of at least $\lceil(1-\alpha)N\rceil$ of the robust sampled  constraints given by
\begin{align}
g(x,u) \leq 0 \quad \forall \, u \in \uncertset,  \quad i = 1, \dots, N, \label{eq:robust sampled constraint}
\end{align}
where $\uncertset$ can be interpreted as the \emph{uncertainty set} associated with each constraint.

Naturally, the radii of uncertainty sets $\uncertset$ can be designed to compensate for  the worst-case discrepancy between  the sequence of  sampling distributions $\{\dist_i\}_{i=1}^N$  and the target distribution $\dist_{N+1}$.  Intuitively,  the  target random variable $\rand_{N+1}$ will lie within the uncertainty set $\uncertset$ with high probability if its radius $r_i$ is sufficiently large. Building on this intuition, we prove the following result,
which relates the robust constraint satisfaction probabilities to the target constraint satisfaction probability.
\begin{lem}\label{lemma:intermediate result}
Let $x \in \Xcal$. For all $i \in [N]$, it holds that
\begin{align*}
&\prob{ g(x,u) \leq 0 \ \forall \, u \in  \uncertset}  \\ 
& \hspace{.8in}  \leq \prob{g(x,\rand_{N+1}) \leq 0}  \,  +  \, \frac{\rho(N+1-i)}{r_i}.
\end{align*}
\end{lem}

\begin{proof}
Let $\randa_1, \dots, \randa_N$ be a sequence of i.i.d. random variables, where 
each random variable in the sequence is identically distributed to $\rand_{N+1}$. 
Also, for all $i\in [N]$, let each pair of random variables $(\rand_i, \, \randa_i)$ be coupled according to the joint distribution that attains the 1-Wasserstein distance between their respective distributions $\dist_i$ and $\dist_{N+1}$. As we assume that $\Xi$ is a Polish space, such an optimal coupling exists \cite[Proposition 1.3]{wang2012coupling}. This, together with \eqref{eq:dist}, implies 
\begin{equation}
\mathbb{E}\left[ \,   \| \rand_i - \randa_i\|  \,\right]  \, = \, \wass{\dist_i}{\dist_{N+1}} \, \leq \, \rho(N+1-i). \nonumber
\end{equation}
By the law of total probability, we have that, for all $i \in [N]$ and $x \in \Xcal$,
\begin{align}
& \hspace{-12pt}\prob{ g(x,u) \leq 0 \ \forall \, u \in  \uncertset} \nonumber \\
    &=\prob{g(x,u)\leq 0 \ \forall \, u \in  \uncertset, \  \|\rand_i-\randa_i\|\leq r_i } \nonumber\\
    &\phantom{=\,} + \ \prob{g(x,u)\leq 0 \ \forall \, u \in  \uncertset, \ \|\rand_i-\randa_i\|> r_i} 
    \nonumber \\
    &\leq \prob{g(x,\randa_i) \leq 0} + \prob{\|\rand_i-\randa_i\|> r_i } \nonumber \\
    & \leq \prob{g(x,\randa_i) \leq 0} +\frac{\mathbb{E}\left[ \,   \| \rand_i - \randa_i\|  \,\right] }{r_i} \nonumber \\
    &\leq \prob{g(x,\rand_{N+1}) \leq 0} +\frac{\rho(N+1-i)}{r_i}
    \nonumber,
\end{align}
where the second to last inequality follows from the application of Markov's inequality. The last inequality follows from the facts that $\randa_i$ is identically distributed to $\rand_{N+1}$ and $\mathbb{E}\left[ \,   \| \rand_i - \randa_i\|  \,\right] \leq \rho(N+1-i) $,  completing the proof.
\end{proof}

Before presenting the main result of this section, it is helpful to define the family of Poisson binomial random variables. 
\begin{defn} A \emph{Poisson  binomial  random variable} is defined as a finite sum  $\mathbf{z} = \sum_{i=1}^N \mathbf{z}_i$, where $\mathbf{z}_1, \dots, \mathbf{z}_N$ are  independent Bernoulli random variables with expectations $\mathbb{E}[\mathbf{z}_i] = q_i$. Its cumulative distribution function  is given by 
\begin{align*}
 \prob{\mathbf{z} \leq z}  = \sum_{k=0}^{\lfloor z \rfloor} \sum_{ A \in S_k}  \left( \prod_{i \in A} q_i \prod_{i \notin A} (1-q_i) \right), \quad z \in [0,N],
\end{align*}
where $S_k$ is the set of all subsets of $[N]$ of cardinality $k$. We denote its cumulative distribution function  by $\Psi \left(z; q_1, \dots, q_N\right) :=   \prob{\mathbf{z} \leq z}$ for $z \in [0,N]$.
\end{defn}

\subsection{Finite $\mathcal{X}$}\label{subsec:nonstat finite X}
The main result of this  paper is given in Theorem \ref{theorem:robust SAA for finite X}, providing
an upper bound on the \emph{probability of infeasibility} for finite $\Xcal$ in nonstationary environments.

\begin{thm}\label{theorem:robust SAA for finite X}
Suppose that $|\mathcal{X}| < \infty$. Then
\[
\prob{ \robfinset \nsubseteq \mathcal{X}_\epsilon^{N+1}} \leq |\mathcal{X}| \Psi(\alpha N; p_1,\ldots, p_N),
\]
where $$p_i:=\left(\epsilon-\frac{\rho(N+1-i)}{r_i} \right)_+ $$
 for $i =1,\ldots,N$.
\end{thm}
\begin{proof}
Let $\overline{\mathcal{X}}_\epsilon^{N+1}:= \mathcal{X} \setminus \mathcal{X}_\epsilon^{N+1}$. We have that
\begin{align} 
& \prob{ \robfinset \nsubseteq \mathcal{X}_\epsilon^{N+1}} \nonumber \\
&=\prob{ \exists x \in \overline{\mathcal{X}}_\epsilon^{N+1} \text{ such that } x \in \robfinset} \nonumber \\
&\leq \sum_{x \in \overline{\mathcal{X}}_\epsilon^{N+1}} \prob{\robfinviol(x) \leq \alpha} \nonumber \\
 &=\sum_{x \in \overline{\mathcal{X}}_\epsilon^{N+1}} 
\prob{ \sum_{i=1}^N \indi{\exists u \in \uncertset, \, g(x,u)\!>\!0} \leq \alpha N} \nonumber \\
&=\sum_{x \in \overline{\mathcal{X}}_\epsilon^{N+1}} \Psi(\alpha N;v_1(x),\ldots, v_N(x)). \nonumber
\end{align}
The final equality follows from the fact that $\sum_{i=1}^N \indi{\exists u \in \uncertset, \, g(x,u)>0}$ is a Poisson binomial random variable. Specifically, it is the sum of $N$ independent Bernoulli random variables with heterogeneous success probabilities  given by $v_i(x) := \prob{\exists u \in \uncertset, \, g(x,u)>0}$ for $i=1,\ldots,N$. Note that, for all $x \in \comp^{N+1}$ and $i\in [N]$, it holds that 
\begin{align*}
   v_i(x) &= 1-\prob{g(x,u)\leq 0 \ \forall \, u \in  \uncertset} \\
    &\geq 1-\prob{g(x,\rand_{N+1}) \leq 0}  - \frac{\rho(N+1-i)}{r_i}\\
    &>\risk - \frac{\rho(N+1-i)}{r_i}.
\end{align*}
The first inequality follows from Lemma \ref{lemma:intermediate result} since $\comp^{N+1} \subseteq \Xcal$. The second inequality follows from  $x \in \comp^{N+1}$.
Therefore,
\[ 
v_i(x) \geq p_i = \left( \epsilon - \frac{\rho(N+1-i)}{r_i} \right)_+
\]
for all $i \in [N]$ and $x \in \overline{\mathcal{X}}_\epsilon^{N+1}$.
The claim follows, as
\begin{align}
\nonumber \prob{\robfinset \nsubseteq \ccset^{N+1}} & \leq  \sum_{ x \in  \overline{\Xcal}_\risk^{N+1} } \Psi \left(\alpha N; v_1(x), \dots, v_N(x)\right)\\
\nonumber & \leq  \sum_{ x \in  \overline{\Xcal}_\risk^{N+1} } \Psi \left(\alpha N; p_1, \dots, p_N\right)\\
\nonumber & \leq \  |\Xcal| \Psi \left(\alpha N; p_1, \dots, p_N\right). 
\end{align}  
\end{proof}

\begin{rem}[Choosing the Uncertainty Sets] Each of the uncertainty sets specified in \eqref{eq:uncertainty set} is defined as a norm-ball intersected with the support set $\Xi$. An important consequence of this definition is that the feasible set of the robust SAA defined in \eqref{eq:robust saa} is guaranteed to contain the so-called \emph{robust feasible set} associated with this problem, that is, 
\begin{align} \label{eq:contain}
   \robfinset \supseteq \{ x \in \Xcal \ | \ g(x,u) \leq 0 \ \forall \, u \in \Xi\}.
\end{align}
It follows from \eqref{eq:contain} that nonemptiness of the robust feasible set implies nonemptiness of the robust SAA feasible set.  However, defining the uncertainty sets in this manner requires knowledge of the underlying support set. If 
this information
is not available, then the  uncertainty sets in \eqref{eq:robust sampled constraint} can  be redefined as norm-balls $U_{r_i}(\rand_i):=\{ u \in \Rset^d  \, | \,  \| u- \rand_i \| \leq r_i\}$ without intersecting them  with the support set. Although the resulting robust SAA will yield solutions that are potentially more conservative, the claim in Theorem \ref{theorem:robust SAA for finite X} continues to hold true. 
\end{rem}

\begin{rem}[Tractability of the Robust Approximation] The robust sampled constraints defined in \eqref{eq:robust sampled constraint} admit computationally tractable reformulations for a large family of constraint functions and uncertainty sets. For example, if the constraint function is a bi-affine function of the form $g(x,u) = x^\top u - b$ and the uncertainty sets can be described as convex polytopes (e.g., if  $\Xi$ is a convex polytope and $\|\cdot\|$ is the  $L^1$ or $L^\infty$-norm), then each of the robust constraints  in \eqref{eq:robust sampled constraint} can be reformulated as a finite collection of linear constraints. We refer the reader to \cite{bertsimas2011theory} for a  comprehensive discussion about  the families of robust constraints that admit equivalent reformulations as tractable convex constraints.
\end{rem}

\begin{rem}[Choice of Radii] 
Theorem \ref{theorem:robust SAA for finite X} allows for flexibility in how to choose the uncertainty set radii. 
A natural choice is to let each radius $r_i$ be proportional to the Wasserstein distance between the sampling distribution $\dist_i$ and the target distribution $\dist_{N+1}$, that is,
\begin{equation}\label{eq:radii}
r_i = \rho(N + 1 -i)/\theta
\end{equation}
for some constant $\theta > 0$. 
We note that this choice of radii has an intuitive interpretation:  older (less informative) samples are  associated with uncertainty sets with larger radii to account for the greater potential discrepancy between their distributions and the target distribution.  Therefore, it is natural to expect that robust constraints associated with older samples (larger uncertainty sets) are more likely to be discarded when solving the resulting robust SAA, as their exclusion will frequently result in the greatest enlargement of the feasible region $\robfinset$. Note that for $\alpha > 0$,  at most $\lfloor \alpha N \rfloor $ constraints can be discarded. 
\end{rem}

By using Hoeffding's inequality \cite{Hoeffding1963inequality}, one can upper bound the tail of the Poisson binomial distribution in Theorem \ref{theorem:robust SAA for finite X} as
\[ \textstyle 
\Psi(\alpha N; p_1,\ldots,p_N) \leq \exp \left( -2 N\left( \frac{1}{N}\sum_{i=1}^N p_i-\alpha  \right)^2 \right)
\]
if $(1/N)\sum_{i=1}^N p_i > \alpha$. This  upper bound can be used to  characterize
a distribution-free lower bound on the sample size $N$ needed to  ensure that $\robfinset$ is contained within $\Xcal_\epsilon^{N+1}$ with a given confidence $1-\delta$. Corollary  \ref{cor:sample bound} illustrates this procedure for the choice of radii given in \eqref{eq:radii}.

\begin{cor}\label{cor:sample bound} Let $\delta \in (0,1)$, $\alpha \in (0, \risk)$, and  $\theta \in (0, \risk - \alpha)$. If the uncertainty set radii are specified as $r_i = \rho(N + 1 -i)/\theta$ for all $i \in [N]$ and  
\begin{equation}
N \geq \frac{1}{2(\epsilon -\alpha - \theta)^2} \ln \left( \frac{|\mathcal{X}|}{\delta} \right), 
\end{equation}
then $\prob{ \robfinset \subseteq \mathcal{X}^{N+1}_\epsilon} \geq 1-\delta$.
\end{cor}

\subsection{Lipschitz Continuous $g$}\label{subsec:nonstat infinite X}
In this section, we propose a slight modification of the robust SAA to address problems where the set $\Xcal$ may not be finite. We modify the approximation by incorporating a constraint tightening parameter as in Section \ref{subsec:infinite X in stationary}. Using the modified approximation, we provide an upper bound on the \emph{probability of infeasibility} that holds under Assumptions \ref{ass:lip} and \ref{ass:bound}. 
The modified robust SAA of $\Xcal_\risk^{N+1}$ is defined as
\begin{equation} 
\robinfset:= \{ x \in \Xcal \ |\ \robinfviol(x) \leq \alpha \} \nonumber ,
\end{equation} 
where  $\gamma \in \Rset_+$,  $r=(r_1,\ldots,r_N)\in \Rset_+^N$, and 
\[\robinfviol(x) \!:=\!\frac{1}{N} \!\sum_{i=1}^N \indi{ \exists \, u \in \uncertset \text{ such that } g(x,u)\!+\!\gamma\!>\!0   }. \]
\begin{thm}\label{theorem:robust SAA for infinite X}
Suppose that Assumptions \ref{ass:lip} and \ref{ass:bound} hold. 
Then
\begin{equation*}
\prob{\robinfset \nsubseteq \mathcal{X}^{N+1}_\epsilon } \leq \left( \frac{LD}{\gamma} +1 \right)^n \Psi(\alpha N; p_1,\ldots,p_N), \label{eq:error prob bound in nonstationary for infinite X}
\end{equation*}
where  $$p_i:=\left( \epsilon- \frac{\rho(N+1-i)}{r_i}\right)_+ $$ for $i=1,\ldots, N$.
\end{thm}

We omit the proof of Theorem \ref{theorem:robust SAA for infinite X}, as it is straightforward to show by combining the arguments used in proving Theorems \ref{thm:main} and \ref{theorem:robust SAA for finite X}.
 \section{Conclusion} \label{sec:conclusion}
We investigate sample average approximations for chance constraints in both stationary and nonstationary environments. For stationary environments, we provide a new upper bound on the \emph{probability of infeasibility} for the setting in which $\Xcal$ has possibly infinite cardinality. For nonstationary environments,  we propose a robust SAA scheme in which each random sample encodes a robust sampled constraint defined in terms of an uncertainty set whose radius depends on the distributional uncertainty of the associated random sample. We derive upper bounds on the corresponding \emph{probability of infeasibility} for this class of approximations.

\bibliographystyle{IEEEtran}
\bibliography{references}{\markboth{References}{References}}

\end{document}